\documentclass[a4paper,12pt]{article}
\usepackage{amsmath}
\usepackage{amssymb}
\usepackage{amsthm}
\usepackage{authblk}
\usepackage{setspace}
\usepackage[top=43truemm,bottom=43truemm,left=30.9truemm,right=30.9truemm]{geometry}
\usepackage{amsfonts}
\usepackage{amscd}
\numberwithin{equation}{section}
\newtheorem{mainresult}{Main}

\newtheorem{example}{Example}[section]

\newtheorem{theorem}[example]{Theorem}
\newtheorem{definition}[example]{Definition}
\newtheorem{lemma}[example]{Lemma}
\newtheorem{claim}{Claim}
\newtheorem{remark}[example]{Remark}
\newtheorem{question}[example]{Question}
\DeclareMathOperator{\spec}{Spec}%
\begin{document}
\setstretch{0.97}
\title{Counter-examples to 
non-noetherian \\ Elkik's approximation theorem}
\author{Kei Nakazato}
\date{}
\maketitle
\begin{abstract}
Elkik established a remarkable theorem that can be applied for any noetherian henselian ring. For algebraic equations with a formal solution (restricted by some smoothness assumption), this theorem provides a solution adically close to the formal one 
in the base ring. 
In this paper, we show that the theorem would fail for some non-noetherian henselian rings. These rings do not satisfy several conditions weaker than noetherianness, such as weak proregularity (due to Grothendieck et al.) of the defining ideal. We describe the resulting pathologies. 
\end{abstract}
\section{Introduction}
\label{intro}
The main goal of this paper is to show that \emph{Elkik's approximation theorem} (Theorem \ref{ElkikT}) would fail in some non-noetherian cases. 
Elkik's approximation theorem is used for giving affirmative answers to a fundamental question in M.\ Artin's celebrated work \cite{Art1}. We first recall it: 
\begin{question}
[cf.\ {\cite[Question 1.7]{Art1}}]
\label{qsart}
Let $(A, I)$ be a pair consisting of a ring $A$ and its ideal $I$. Set $\widehat{A}$ to be the $I$-adic completion of $A$. Let $f=(f_{1},\ldots, f_{m})$ be a polynomial system in $A[X_{1},\ldots, X_{N}]$, and suppose that the equation system $f=0$ (i.e.\ $f_{1}=0,\ldots, f_{m}=0$) has a solution 
$\widehat{{\boldsymbol{\alpha}}}=(\widehat{\alpha}_{1},\ldots,\widehat{\alpha}_{N})\in{\widehat{A}}^{N}$. 
Let $c$ be a positive integer. Does there exist a solution $\boldsymbol{\alpha}=(\alpha_{1},\ldots, \alpha_{N})\in A^{N}$ of $f=0$ such that 
$\alpha_{i}\equiv\widehat{\alpha}_{i}\mod I^{c}\widehat{A}\ (i=1,\ldots, N)$?
\end{question}
For an important class of henselian pairs, Artin proved that each pair has the following property
 (called the Artin approximation property): for every equation system, every solution, and every $c>0$, Question \ref{qsart} has an affirmative answer (\cite[Theorem 1.10]{Art1}). Artin's result has been generalized to a far stronger form below (cf.\ \cite{And}, \cite{Ogo}, \cite{P85}, \cite{P86}, \cite{P90}, and \cite{Swa}): if $(A, I)$ is noetherian\footnote{In \cite{Mor} and \cite{Sch}, it is shown that some interesting non-noetherian pairs have the Artin approximation property.} (i.e.\ $A$ is noetherian) and henselian, and the natural morphism $A\to \widehat{A}$ is regular, then $(A, I)$ has the Artin approximation property. 

In Elkik's theorem, the equations and the solutions are restricted by some smoothness assumption, which reflects a good lifting property of a henselian pair described in \cite[Lemme 2]{Elk}. By virtue of this, the theorem can be applied for a far broader class of henselian pairs, including all noetherian ones. We denote by $(*)$ the condition imposed on henselian pairs. For $(A, I)$, $(*)$ means either one of the following conditions (cf.\ \S\ref{section2.2}):
\begin{itemize}
\item[$\bold{({E}_{a})}$]{\emph{$I$ is principal, and $A$ has bounded $I$-torsion (i.e.\ there exists an integer $l>0$ such that $I^{l}A_{I\textnormal{-tor}}=(0)$, cf.\ \S\ref{tmpr})}};
\item[$\bold{({E}_{b})}$]{\emph{there exists a noetherian pair $(A_{0}, I_{0})$ such that $A$ is flat over $A_{0}$ and $I=I_{0}A$}. }
\end{itemize}

The basic purpose of this paper is to investigate what part of $(*)$ is essential for Elkik's theorem. The following is our main result, which particularly claims that the bounding condition on $A_{I\textnormal{-tor}}$ is crucial when $I$ is principal, but cannot be substituted for $(*)$. 
\begin{mainresult}
In Elkik's approximation theorem (Theorem \ref{ElkikT}), $(*)$ cannot be weakened to either one of the following conditions:
\begin{itemize}
\item[$\bold{({{E}'_{a}})}$]{$I$ is principal;}
\item[$\bold{({{E}''_{a}})}$]{$I$ is finitely generated, and $A$ has bounded $I$-torsion;}
\item[$\bold{({{E}'_{b}})}$]{there exists a pair $(A_{0}, I_{0})$ such that $A_{0}$ is noetherian outside $I_{0}$ {(}i.e.\ the scheme $\spec{A_{0}}\setminus V(I_{0})$ is noetherian{)}, $A$ is flat over $A_{0}$, and $I=I_{0}A$}. 
\end{itemize}
\end{mainresult}
We prove this statement by giving two examples in \S\ref{section3.1}. 
The construction is based on Greco and Salmon's example of non-flat $I$-adic completion (where $I=tA$). This example gives a negative answer to Question \ref{qsart} (Example \ref{GSKPR}), but it lacks the smoothness assumption that we require. To overcome this, we define an auxiliary equation, and impose suitable relations on $A$. These relations produce an element $\xi_{n}\in A$ for every integer $n\geq 1$, such that $t^{n}\xi_{n}\neq 0$ but $t^{n+1}\xi_{n}=0$ (and then $(*)$ is no longer satisfied). The relation $t^{n+1}\xi_{n}=0$ contributes to solving the additional equation in $\widehat{A}$, and the other $t^{n}\xi_{n}\neq 0$ is required to conclude the algebraic approximation fails. 

On the course of this study, it was found\footnote{Liran Shaul informed the author (cf.\ Acknowledgements). } that Elkik's theorem is related to \emph{weak proregularity}, a notion originated in Grothendieck's work\footnote{The term ``weakly proregular'' is due to \cite[Correction]{AJL}. } \cite{Gro}. For a pair $(A, I)$, weak proregularity of $I$ allows the derived functors of $I$-torsion and $I$-adic completion to behave well (see \cite{PSY} for details). We note that the weak proregularity condition is a generalization of $(*)$. Moreover, when $I$ is principal, they are equivalent. Thus the main result on $\bold{(E'_{a})}$ signifies the importance of weak proregurarity for Elkik's Theorem in this case. On the other hand, one of the henselian pairs given in \S\ref{section3.1} is defined by a non-principal ideal, which is not weakly proregular. These facts might say that weak proregularity is a key notion for Elkik's algebraic approximation. 

The organization of the paper is as follows. In \S\ref{tmpr}, we give
some terminology and preliminary results for later use. In \S\ref{counterexample}, we first recall Greco and Salmon's example, and then construct the principal examples to prove the main result. In \S\ref{subsecwkpr}, we check the statements on weak proregularity given in the preceding paragraph. 
\section{Terminology and preliminaries}
\label{tmpr}
Throughout this paper, all rings are assumed to be commutative with unit. 
By a \emph{pair} we mean a pair $(A, I)$ consisting of a ring $A$ and its ideal $I$. A 
\emph{morphism of pairs} $u:(A, I)\to (B, J)$ means a ring homomorphism $u:A\to B$ such that $u^{-1}(J)=I$. 

Let $A$ be a ring, and $a\in A$ an element. An element $x\in A$ is said to be \emph{$a$-torsion}, if there exists an integer $n>0$ such that $a^{n}x=0$. We denote by $A_{a\textnormal{-tor}}$ the ideal consisting of all $a$-torsion elements in $A$. For an ideal $I\subset A$, we denote by $A_{I\textnormal{-tor}}$ the ideal $\bigcap_{a\in I}A_{a\textnormal{-tor}}$. We say that $A$ has \emph{bounded $I$-torsion}, if there exists an integer $l>0$ such that $I^{l}A_{I\textnormal{-tor}}=(0)$ (or, equivalently, an ascending chain $\textnormal{Ann}_{A}({I})\subset\textnormal{Ann}_{A}({I}^{2})\subset
\textnormal{Ann}_{A}({I}^{3})\subset\cdots$ becomes stationary). 
\subsection{Henselian pairs}
Let us first recall some basic notions on \emph{henselian pairs}. 
There are a number of 
formulations of the definition of a henselian 
pair (cf.\ \cite[2.6.1]{KPR} or \cite[Chapitre XI, $\S${2}]{Ray}), and 
here we give the one adopted in Elkik's paper \cite{Elk}. 
\begin{definition}\normalfont 
We say that a pair $(A, I)$ is \emph{henselian} (or $A$ is \emph{$I$-adically henselian}), if $I$ is contained in the Jacobson radical of $A$, and for every \'{e}tale $A$-algebra $A'$ such that the induced homomorphism 
$A/I\to A'\otimes_{A}A/I\nonumber$ 
is an isomorphism, there exists an $A$-homomorphism $A'\to A$. 
\end{definition}
One can associate any pair with a universal construction of henselian pairs. 
\begin{theorem}[{\cite[6.1(i)]{Gre}}]
\label{henselization}
Let $(A, I)$ be a pair. Then there exists a henselian pair $(A^{h}, I^{h})$ together with a morphism $\phi:(A, I)\to (A^{h}, I^{h})$, which has the following universal property:
\begin{itemize}
\item{
for every henselian pair $(A', I')$ together with a morphism ${\psi}:(A, I)\to (A', I')$, there exists a unique morphism ${\psi}':(A^{h}, I^{h})\to (A', I')$ such that ${\psi}'\circ\phi=\psi$. }
\end{itemize}
\end{theorem}
The pair $(A^{h}, I^{h})$ (which is unique up to canonical isomorphisms) is called the \emph{henselization of $(A, I)$}, and $A^{h}$ is called the \emph{$I$-adic henselization of $A$}. 
We will use the following property of henselization. 
\begin{lemma}[{\cite[6.1(iii)]{Gre}}]
\label{henselproperty}Let $(A, I)$ be a pair, and $(A^{h}, I^{h})$ the henselization of 
$(A, I)$. Then $I^{h}=IA^{h}$, and the morphism $A\to A^{h}$ is flat. 
\end{lemma}
\subsection{Elkik's approximation theorem}
\label{section2.2}
Here we review Elkik's approximation theorem. 
We begin by defining the condition $(*)$ on a pair, which is imposed on a henselian pair in Elkik's theorem. 
\begin{definition}\normalfont
We say that a pair $(A, I)$ satisfies $(*)$, if it satisfies either $\bold{({E}_{a})}$ or $\bold{({E}_{b})}$ 
(these two conditions are given in \S\ref{intro}). 
\end{definition}
Note that $(*)$ is preserved under henselization (see Remark \ref{ebea} below) like noetherianness (cf.\ \cite[6.9(i)]{Gre}). 
\begin{remark}\normalfont
\label{ebea}
If $(A, I)$ satisfies $(*)$, then so does $(B, IB)$ for every flat $A$-algebra $B$ (in particular, so does $(A^{h}, I^{h})$ by Lemma \ref{henselproperty}). To check this, it suffices to consider the case where $(A, I)$ satisfies $\bold{(E_{a})}$. 
In this case, the ideal $IB\subset B$ is principal. Moreover, since $B$ is flat over $A$, it holds that 
$\textnormal{Ann}_{B}(I^{k}B)=(\textnormal{Ann}_{A}({I}^{k}))B$ 
for every $k\geq 1$ (\cite[Theorem 7.4(iii)]{Mat}). Therefore an ascending chain 
$\textnormal{Ann}_{B}(IB)\subset\textnormal{Ann}_{B}({I}^{2}B)\subset
\textnormal{Ann}_{B}({I}^{3}B)\subset\cdots$ becomes stationary. Thus, $(B, IB)$ also satisfies $\bold{(E_{a})}$. 
\end{remark}
We exhibit below some examples of pairs which satisfy $(*)$. 
\begin{example}\normalfont
Clearly noetherian pairs satisfy $(*)$. We list other examples. 
\begin{itemize}
\item[(1)]{Let $V$ be a valuation ring with a non-zero element $t\in V$ which is not invertible. If the value group of $V$ is not isomorphic to $\mathbb{Z}$, then $V$ is not noetherian (\cite[Theorem 11.1]{Mat}). Assume that $V$ is $t$-adically separated. Then, for every finitely generated $V$-algebra $A$, the pair $(A, tA)$ satisfies $\bold{({E}_{a})}$ (\cite[Theorem 7.4.1(1)]{FGK}). Next, assume that $V$ is $t$-adically complete. Let $V\langle X_{1},\ldots, X_{n}\rangle$ be the $t$-adic completion of a polynomial ring $V[X_{1},\ldots, X_{n}]$. Then, for every finitely generated $V\langle X_{1},\ldots, X_{n}\rangle$-algebra $B$, $(B, tB)$ satisfies 
$\bold{({E}_{a})}$ (\cite[Theorem 7.3.2]{FGK}). }
\item[(2)]{Let $R$ be a complete regular local ring of mixed characteristic $p>0$. Let $x_{1},\ldots, x_{d}$ be a regular system of parameters of $R$. Put $R_{n}=R[x_{1}^{1/p^{n}},\ldots, x_{d}^{1/p^{n}}]$ for $n\geq 1$ where $\{x_{i}^{1/p^{n}}\}_{n\geq 1}$ is a system of $p$-power roots of $x_{i}$ such that $({x_{i}^{1/p^{n+1}}})^{p}=x_{i}^{1/p^{n}}$ $(i=1,\ldots, d)$. Let $R_{\infty}$ be the inductive limit of $\{R_{n}\}_{n\geq 1}$. Then, for an ideal $I\subset R$, the pair $(R_{\infty}, IR_{\infty})$ satisfies $\bold{({E}_{b})}$ (this class of rings play an important role in recent works on the homological conjectures, cf.\ \cite{YAn} or \cite{Shi}). }
\end{itemize}
By considering the henselizations, one can also obtain examples of non-noetherian henselian pairs which satisfy $(*)$ (cf.\ \cite[6.9(ii)]{Gre}). 
\end{example}
Now we state the assertion of Elkik's approximation theorem. 
\begin{theorem}[{\cite[Th\'{e}or\`{e}me 2 \emph{bis} and p.587]{Elk}}]
\label{ElkikT}
Let $(A, I)$ be a henselian pair, and suppose that $(A, I)$ satisfies $(*)$. Let $\widehat{A}$ be the $I$-adic completion of $A$, 
$B$ a finitely presented $A$-algebra, and $\overline{B}=B\otimes_{A}\widehat{A}$. Let $V\subset \spec{B}$ be an open set which is smooth over 
$\spec{A}$, and $\overline{V}\subset \spec{\overline{B}}$ the preimage of $V$. 

Then for every integer $n\geq 0$ and every $\widehat{A}$-section 
$\overline{\varepsilon}:\spec{\widehat{A}}\to\spec{\overline{B}}$ 
whose restriction to $\spec{\widehat{A}}\setminus V(I\widehat{A})$ factors through $\overline{V}$, 
there exists an $A$-section ${\varepsilon}:\spec{A}\to\spec{B}$ congruent to $\overline{\varepsilon}$ modulo $I^{n}$, whose restriction to $\spec{A}\setminus V(I)$ factors through $V$. 
\end{theorem}
This theorem provides affirmative answers to Question \ref{qsart} in certain situations. Indeed, the $A$-algebra $B$ in the theorem corresponds to a polynomial system $f=(f_{1},\ldots, f_{m})$ over $A$ such that $B=A[X_{1},\ldots, X_{N}]/(f_{1},\ldots, f_{m})$. The $\widehat{A}$-section $\overline{\varepsilon}$ corresponds to a solution $\widehat{{\boldsymbol{\alpha}}}\in{\widehat{A}}^{N}$ of $f=0$ which defines $\overline{\varepsilon}$. In terms of explicit equations, the theorem claims that there exists a solution ${\boldsymbol{\alpha}}\in{A}^{N}$ of $f=0$ such that 
${\boldsymbol{\alpha}}\equiv\widehat{\boldsymbol{\alpha}}\mod I^{n}\widehat{A}^{N}$ (and the corresponding $A$-section $\varepsilon$ satisfies an additional condition). 
\section{Counter-examples}
\label{counterexample}
Here we describe the principal examples to prove the main result. Before that, let us recall a known example of a henselian pair which provides a negative answer for Question \ref{qsart}. 
All these examples are based on the following key fact:
\emph{the $I$-adic henselization of $A$ is always flat over $A$, but the $I$-adic completion is not so}.
\subsection{Greco and Salmon's example of non-flat completion}
The example is the following. 
\begin{example}[{\cite[pp.24-25]{GS} \textnormal{and} \cite[pp.137-138]{KPR}}]\normalfont
\label{GSKPR}
Let $k$ be a field. We consider a polynomial ring in countably many indeterminates $k[x_{0}, x_{1}, x_{2},\ldots][y]$, and its ideal
\begin{eqnarray}
{{\mathfrak{a}}}=(x_{0}y, x_{0}-x_{1}y, x_{1}-x_{2}y, x_{2}-x_{3}y,\ldots)\ . \nonumber
\end{eqnarray}
Let us denote the quotient $k[y, x_{0}, x_{1},\ldots]/{\mathfrak{a}}$ by $R$. 
Consider the polynomial ring $R[t]$. Let $A$ be the $t$-adic henselization of $R[t]$, and $\widehat{A}$ the $t$-adic completion of $A$. Then, for the linear equation $(t-y)X=0$, 
we can find the solution $\widehat{\alpha}=\sum^{\infty}_{i=0}x_{i}t^{i}$ in $\widehat{A}$. However, $R[t]$ does not have any non-zero solution of this equation, nor does $A$ because of the flatness. Hence this example gives a negative answer to Question \ref{qsart} (and \emph{$\widehat{A}$ is not flat over $A$}). 
\end{example}
Note that $(A, tA)$ satisfies $\bold{(E_{a})}$, and thus this example does not work well as a counter-example to Theorem \ref{ElkikT}. 
The reason is that against the hypothesis the restriction of 
$\overline{\varepsilon}$ to $\spec{\widehat{A}}\setminus V(t\widehat{A})$ cannot factor through the preimage of ${V}\subset \spec{B}$, an open set which is \emph{smooth} over $\spec{A}$. The additional equations (in Example \ref{principalex} and \ref{boundedtorex}) are defined to overcome this difficulty. 
\subsection{Proof of the main result}
\label{section3.1}
First we prove the main result on $\bold{(E'_{a})}$ and $\bold{({{E}'_{b}})}$. 
\begin{example}
\label{principalex}
Let $k$ and $R$ be the same as in Example \ref{GSKPR}. 
Consider the polynomial ring $R[t]$ and its ideal 
\begin{eqnarray}
{\mathfrak{n}}=(x_{0}t^{2}, x_{1}t^{3}, x_{2}t^{4},\ldots)\ .\nonumber
\end{eqnarray}
Let $A_{0}$ be the quotient $R[t]/{\mathfrak{n}}$, $A$ the $t$-adic henselization of $A_{0}$, and $\widehat{A}$ the $t$-adic completion of $A$. Consider the polynomial system $f=(f_{1}, f_{2})$ where
\begin{eqnarray}
\begin{cases}
f_{1}=(t-y)X\\
f_{2}=t^{2}X\label{eqsys} 
\end{cases}
\nonumber
\end{eqnarray}
in $A[X]$ and the solution $\widehat{\alpha}=\sum^{\infty}_{i=0}x_{i}t^{i}\in \widehat{A}$ of $f=0$. Let $B$ be the $A$-algebra $A[X]/(f_{1}, f_{2})$, 
$\overline{B}=B\otimes_{A}\widehat{A}$, $V=\spec{B}\setminus V(tB)$, and $\overline{\varepsilon}:\spec{\widehat{A}}\to\spec{\overline{B}}$ the $\widehat{A}$-section corresponding to $\widehat{\alpha}$. Then this example provides a proof of the main result on $\bold{({E'_{a}})}$ and $\bold{({E'_{b}})}$. 
\end{example}
Let us verify it. Since $A_{0}[\frac{1}{t}]$ is a finitely generated $k$-algebra, $A_{0}$ is noetherian outside $tA_{0}$. Thus, by Lemma \ref{henselproperty}, $(A, tA)$ is a henselian pair which satisfies $\bold{(E'_{a})}$ and $\bold{({E'_{b}})}$. 
By virtue of the auxiliary polynomial $f_{2}$, $V=\spec{B}\setminus V(tB)$ is \'{e}tale over $\spec{A}$. Therefore 
$B$, $V$, and $\overline{\varepsilon}$ clearly satisfy the conditions assumed in Theorem \ref{ElkikT}. Thus it suffices to show that any solution of $f_{1}=0$ in $A$ is not congruent to $\widehat{\alpha}$ modulo $t\widehat{A}$. Since $A$ is flat over $A_{0}$, we can make use of the following argument about $A_{0}$. 
Let ${\mathfrak{c}}\subset A_{0}$ be the kernel of the $A_{0}$-endomorphism 
\begin{eqnarray}
A_{0}\to A_{0},\ \xi\mapsto (t-y)\xi\ . \nonumber
\end{eqnarray}
The following claim is essential in our proof.  
\begin{claim}\label{essentialclaim}
${\mathfrak{c}}\subset tA_{0}$. 
\end{claim}
We assume Claim \ref{essentialclaim} at this moment and deduce that the algebraic approximation fails. Let $\beta\in A$ be a solution of $f_{1}=0$. Since $A$ is flat over $A_{0}$, $\beta$ belongs to ${\mathfrak{c}}A$. By Claim \ref{essentialclaim}, it belongs to $tA$, too. However, $x_{0}\notin tA_{0}$, and thus 
$\sum^{\infty}_{i=0}x_{i}t^{i}\notin t\widehat{A}$. Consequently $\beta$ is not congruent to $\widehat{\alpha}$ modulo $t\widehat{A}$, as required. 

Thus we are reduced to showing Claim \ref{essentialclaim}. As a preparation for this, 
we analyze the structures of $R$ and $A_{0}$. 
\begin{claim}
\label{bclaim1}
As a $k$-vector space, $R$ has the basis $\{x_{0}, x_{1}, x_{2},\ldots\}\cup\{1, y, y^{2},\ldots\}$. 
\end{claim}
\begin{claim}
\label{annclaim2}
Let $i$ be a positive integer. 
The annihilator of $t^{i}\in A_{0}$ in $R$ is 
\begin{eqnarray}
\textnormal{Ann}_{R}(t^{i})=
\begin{cases}
\label{ann}
\ \ \ \ \ \ \ \mspace{3mu}(0)&(i=1)\\
kx_{0}+\cdots+kx_{i-2}&(i\geq 2)\ . 
\end{cases}\nonumber
\end{eqnarray}
\end{claim}
\begin{proof}[Proof of Claim\ \ref{bclaim1}]
For non-negative integers $i, j$, it holds that 
\begin{eqnarray}
\label{xxrelation}
x_{i}x_{j}=x_{i+j+1}y^{j+1}x_{j}=0
\end{eqnarray}
in $R$. Hence $R$ is generated by $\{x_{0}, x_{1}, x_{2},\ldots\}\cup\{1, y, y^{2},\ldots\}$ as a $k$-vector space. 
We show that they are linearly independent. 
Let $n\geq 0$ be an integer and assume that 
\begin{eqnarray}\label{cl1}
\sum^{n}_{i=0}e_{-i-1}y^{i}+\sum^{n}_{i=0}e_{i}x_{i}=0 
\end{eqnarray}
where $e_{i}\in k$ $(-n-1\leq i \leq n)$. 
Multiplying both sides of (\ref{cl1}) by $x_{n}$, we get 
$\sum^{n}_{i=0}e_{i-n-1}x_{i}=0$. Thus the linear independence follows from the statement 
\begin{eqnarray}
\label{LI}
\sum^{n}_{i=0}e_{i}x_{i}=0\ \Rightarrow \ e_{n}=0\ . 
\end{eqnarray}
Let us prove (\ref{LI}). Since $y^{n}(\sum^{n}_{i=0}e_{i}x_{i})=e_{n}x_{0}$, it suffices to check that the monomial $x_{0}\in k[x_{0}, x_{1}, x_{2},\ldots][y]$ does not belong to 
${\mathfrak{a}}$. This is easy. 
\end{proof}
\begin{proof}[Proof of Claim\ \ref{annclaim2}]
Note that $R[t]$ has a grading $R[t]=\bigoplus_{d\in\mathbb{N}}S_{d}$, where $\mathbb{N}$ denotes the semigroup of non-negative integers and $S_{d}\subset R[t]$ consists of all monomials of degree $d$. 
For an element $a\in R$, $a\in\textnormal{Ann}_{R}(t^{i})$ means that $at^{i}\in {\mathfrak{n}}\cap S_{i}$ in $R[t]$. Thus, if $i=1$, clearly $a=0$. Suppose that $i\geq 2$. Then 
\begin{eqnarray}
{\mathfrak{n}}\cap S_{i}&=&S_{i-2}x_{0}t^{2}+\cdots+S_{0}x_{i-2}t^{i}\nonumber\\
&=&Rx_{0}t^{i}+\cdots+Rx_{i-2}t^{i}\ .\nonumber
\end{eqnarray}
Hence $a\in\textnormal{Ann}_{R}(t^{i})$ if and only if $a\in Rx_{0}+\cdots+Rx_{i-2}$. The latter ideal is equal to $kx_{0}+\cdots+kx_{i-2}$ by (\ref{xxrelation}), as claimed. 
\end{proof}
We then start to prove Claim \ref{essentialclaim}. 
\begin{proof}[Proof of Claim \ref{essentialclaim}]
Take an element 
$\gamma\in {\mathfrak{c}}$. Then there exist an integer $N\geq 2$ and $c_{0}, c_{1},\ldots, c_{N}\in R$ such that $\gamma=\sum^{N}_{i=0}c_{i}t^{i}$ (a way of choosing $\{c_{i}\}$ is not unique, but it is permissible to choose them arbitrarily). 
Our goal in this proof is to show that $c_{0}=0$. 
Multiplying $\gamma$ by $(t-y)$, we obtain
\begin{eqnarray}
-c_{0}y+\sum^{N-1}_{i=0}(c_{i}-c_{i+1}y)t^{i+1}+c_{N}t^{N+1}=0\ . \label{gamma0}
\end{eqnarray}
Recall that $R[t]$ is a graded ring as in the proof of Claim \ref{annclaim2}. Since ${\mathfrak{n}}\subset R[t]$ is a homogeneous ideal, the quotient $A_{0}$ is also graded by degree of $t$. Therefore we can decompose (\ref{gamma0}) into the following equations about $c_{i}$:
\begin{eqnarray}
&\ &c_{0}y=0\ , \label{2.1}\\
&\ &(c_{i}-c_{i+1}y)t^{i+1}=0\ \ \ \ (i=0, 1,\ldots, N-1)\ , \label{2.2}\\
&\ &c_{N}t^{N+1}=0\label{2.3}\ .  
\end{eqnarray}
From Claim \ref{bclaim1} and (\ref{2.1}), it is easily seen that there exists $d_{0}\in k$ which satisfies $c_{0}=d_{0}x_{0}$. Thus it suffices to verify $d_{0}=0$. Then it is enough to prove that $c_{N}$ admits the expression
\begin{eqnarray}
c_{N}=\sum^{N-1}_{j=0}e^{(N-1)}_{j}x_{j}+d_{0}x_{N}\ , \label{nexp}
\end{eqnarray}
where $e^{(N-1)}_{j}\in k\ (j=0,\ldots, N-1)$. Indeed, by substituting the right-hand side of (\ref{nexp}) for $c_{N}$ in (\ref{2.3}), we first obtain $d_{0}x_{N}t^{N+1}=0$. 
Then $d_{0}x_{N}\in kx_{0}+\cdots +kx_{N-1}$ by Claim \ref{annclaim2}. Thus Claim \ref{bclaim1} implies $d_{0}=0$. 

Now we show (\ref{nexp}) by induction. 
Let us start with $c_{1}$. By the equation with $i=0$ in (\ref{2.2}), $(c_{0}-c_{1}y)t=0$. 
Thus, 
\begin{eqnarray}
&\ \ &(d_{0}x_{1}-c_{1})yt=0\ \ \ (\textnormal{since}\ c_{0}=d_{0}x_{0})\nonumber\\
&\Rightarrow&(d_{0}x_{1}-c_{1})y=0\ \ \ \ \ \ \mspace{1mu}(\textnormal{by\ Claim\ \ref{annclaim2}})\nonumber\\
&\Rightarrow&d_{0}x_{1}-c_{1}\in kx_{0}\ \ \ \ \ \  \ \mspace{2mu}(\textnormal{by\ Claim\ \ref{bclaim1}})\ . \nonumber
\end{eqnarray}
Hence there exists $e^{(0)}_{0}\in k$ such that
\begin{eqnarray}
c_{1}=e^{(0)}_{0}x_{0}+d_{0}x_{1}\ .\label{2.5}
\end{eqnarray}
Secondly, we compute $c_{2}$ by (\ref{2.5}). By the equation with $i=1$ in (\ref{2.2}), $(c_{1}-c_{2}y)t^{2}=0$. Thus, 
\begin{eqnarray}
&\ \ &(e^{(0)}_{0}x_{1}+d_{0}x_{2}-c_{2})yt^{2}=0\ \ \ \ \ \ \ \ \ \ (\textnormal{by}\ (\ref{2.5}))\nonumber\\
&\Rightarrow&(e^{(0)}_{0}x_{1}+d_{0}x_{2}-c_{2})y\in kx_{0} \ \ \ \ \ \ \ \mspace{3mu}(\textnormal{by\ Claim\ \ref{annclaim2}})\nonumber\\
&\Rightarrow&e^{(0)}_{0}x_{1}+d_{0}x_{2}-c_{2}\in kx_{0}+kx_{1} \ \ \ (\textnormal{by\ Claim\ \ref{bclaim1}})\ . \nonumber
\nonumber
\end{eqnarray}
Hence there exist $e^{(1)}_{0}, e^{(1)}_{1}\in k$ such that 
\begin{eqnarray}
c_{2}=e^{(1)}_{0}x_{0}+e^{(1)}_{1}x_{1}+d_{0}x_{2}\ . \nonumber
\end{eqnarray}
For each $i=1,\dots, N-1$, the procedure for computing $c_{i+1}$ from the value of $c_{i}$ is similar to the above. 
As a consequence, we obtain (\ref{nexp}). 
\end{proof}
\begin{remark}\normalfont
For integers $n\geq m\geq 2$, we set $f^{(n)}_{2}\in A[X]$ to be the polynomial $f^{(n)}_{2}=t^{n}X$, and ${\mathfrak{n}^{(m)}}\subset R[t]$ the ideal ${\mathfrak{n}^{(m)}}=(x_{0}t^{m}, x_{1}t^{m+1}, x_{2}t^{m+2},\ldots)$. 
It is possible to replace $(f_{2}, {\mathfrak{n}})$ in Example \ref{principalex} with $(f^{(n)}_{2}, {\mathfrak{n}}^{(m)})$. 
Thus there exist infinitely many counter-examples of this type, and Example \ref{principalex} represents them. 
\end{remark}
Next we give a proof of the main result on $\bold{({{E}''_{a}})}$. 
\begin{example}
\label{boundedtorex}
Let $k$ and $R$ be the same as in Example \ref{GSKPR}. Consider the polynomial ring $R[t, u]$ and its ideals 
\begin{eqnarray}
\nonumber
\begin{cases}
{\mathfrak{n}}=(x_{0}t^{2}, x_{1}t^{3}, x_{2}t^{4}, \ldots)\\
{\mathfrak{n}'}=(x_{0}u, x_{1}tu, x_{2}t^{2}u, x_{3}t^{3}u,\ldots)\ .
\end{cases}
\end{eqnarray}
Put $A'_{0}=R[t, u]/(\mathfrak{n}+{\mathfrak{n}}')$ and $I_{0}=(t, u)A'_{0}$. 
Let $A'$ be the $I_{0}$-adic henselization of $A'_{0}$, $I\subset A'$ the ideal $I_{0}A'$, and $\widehat{A'}$ the 
$I$-adic completion of $A'$. 
Consider the polynomial system $f'=(f_{1}, f_{2}, f_{3})$ where
\begin{eqnarray}
\begin{cases}
f_{1}=(t-y)X\\
f_{2}=t^{2}X\\
f_{3}=uX
\nonumber
\end{cases}
\end{eqnarray}
in $A'[X]$ and the solution $\widehat{\alpha}=\sum^{\infty}_{i=0}x_{i}t^{i}\in \widehat{A'}$ of $f'=0$. Let $B'$ be the $A'$-algebra $A'[X]/(f_{1},f_{2},f_{3})$, 
$\overline{B'}=B'\otimes_{A'}\widehat{A'}$, $V'=\spec{B'}\setminus V(IB')$, and $\overline{{\varepsilon}'}:\spec{\widehat{A'}}\to \spec{\overline{B'}}$ the $\widehat{A'}$-section corresponding to $\widehat{\alpha}$. Then this example provides a proof of the main result on $\bold{({E''_{a}})}$. 
\end{example}
The verification is similar to that of Example \ref{principalex}. To demonstrate the failure of the approximation, 
it suffices to show the following statement which corresponds to Claim \ref{essentialclaim}. 
\begin{claim}
\label{stcorrc3}
$\{\gamma\in A'_{0}\ |\ f_{1}(\gamma)=0\}\subset I_{0}$. 
\end{claim}
\begin{proof}[Proof of Claim\ \ref{stcorrc3}]
We regard $A'_{0}$ as a graded ring, as follows. For $(i, j)\in \mathbb{N}^{2}$, set 
$S_{(i,j)}=\{at^{i}u^{j}\in R[t,u]\ |\ a\in R\}$. 
Then we have the graded decomposition $R[t,u]=\bigoplus_{(i,j)\in\mathbb{N}^{2}}S_{(i,j)}$. 
Since the ideal $\mathfrak{n}+{\mathfrak{n}}'\subset R[t,u]$ is homogeneous, the quotient $A'_{0}$ is also $\mathbb{N}^{2}$-graded. Moreover, by calculating $({\mathfrak{n}}+{\mathfrak{n}'})\cap S_{(i,j)}$, we can verify that the annihilator of $t^{i}u^{j}\in A'_{0}$ in $R$ is
\begin{eqnarray}
\textnormal{Ann}_{R}(t^{i}u^{j})=
\begin{cases}
\label{ann2}
\ \ \ \ \ \ \ \ (0)&(i=1,\ j=0)\\
kx_{0}+\cdots+kx_{i-2}&(i\geq 2,\ j=0)\\
kx_{0}+\cdots+kx_{i}&(i\geq 0,\ j\geq 1)
\end{cases}
\end{eqnarray}
(cf.\ the proof of Claim \ref{annclaim2}). 

Take an element $\gamma\in A'_{0}$, and set ${\mathbb{N}}^{2}_{\leq{n}}=\{(i, j)\in \mathbb{N}^{2}\ |\ i\leq n,\ j\leq n\}$ for $n\in \mathbb{N}$. Then $\gamma$ can be expressed as 
$\gamma=\sum_{({i,j)}\in{\mathbb{N}^{2}_{\leq{N}}}}c_{i,j}t^{i}u^{j}$ 
where $N\geq 2$ and $c_{i,j}\in R$. Put $c'_{i}=c_{i,0}$. 
If $f_{1}(\gamma)=0$, from the grading of $A'_{0}$, we can find the following equations about $c'_{i}$:
\begin{eqnarray}
&\ &c'_{0}y=0\ , \nonumber\\
&\ &(c'_{i}-c'_{i+1}y)t^{i+1}=0\ \ \ \ (i=0, 1,\ldots, N-1)\ , \nonumber\\
&\ &c'_{N}t^{N+1}=0\nonumber\ .  
\end{eqnarray}
Then by the latter argument in the proof of Claim \ref{essentialclaim}, we obtain $c'_{0}=0$. Hence 
$\gamma\in I_{0}$. 
\end{proof}
Lastly, let us check that $(A', I)$ satisfies $\bold{({{E''_{a}}})}$. 
\begin{claim}
\label{btcl5}
$A'$ has bounded $I$-torsion. 
\end{claim}
\begin{proof}[Proof of Claim\ \ref{btcl5}]
Take an element $\gamma\in (A'_{0})_{I_{0}\textnormal{-tor}}$. Then it can be expressed as 
$\gamma=\sum_{({i,j)}\in\mathbb{N}^{2}_{\leq{N}}}c_{i,j}t^{i}u^{j}$, as in the proof of Claim \ref{stcorrc3}. Since $\gamma$ is $u$-torsion, there exists an integer $l>0$ such that $c_{i,j}t^{i}u^{j+l}=0$ for every $(i,j)\in\mathbb{N}^{2}_{\leq{N}}$ (by the grading of $A'_{0}$). Thus $c_{i,j}\in kx_{0}+\cdots+kx_{i}$ by (\ref{ann2}). Therefore $c_{i,j}t^{i}u^{j}$ is annihilated by $t^{2}$ and $u$. 
Thus, ${I_{0}}^{2}(A'_{0})_{I_{0}\textnormal{-tor}}=(0)$. Here $A'$ is flat over $A'_{0}$, and therefore we can apply \cite[Theorem 7.4(i$\mspace{-1mu}$i$\mspace{-1mu}$i)]{Mat} as in Remark \ref{ebea}. Consequently we obtain $I^{2}A'_{I\textnormal{-tor}}=(0)$, as claimed. 
\end{proof}
\section{Remarks on weak proregularity}
\label{subsecwkpr}
Here we give several remarks on weak proregularity and the examples in \S\ref{section3.1}. We refer \cite{PSY} and \cite{Szl} for basic properties of weak proregularity, and use the following notation. 
Let $A$ be a ring, $\underline{a}=a_{1},\ldots, a_{n}$ a sequence of elements in $A$, and $M$ an $A$-module. $K(\underline{a})$ denotes the Koszul complex of $\underline{a}$ (cf.\ \cite[\S{4.5}]{Wei}), and we set $K(\underline{a};M)$ to be the chain complex $K(\underline{a})\otimes_{A} M$. Let $H_{k}(\underline{a})$ (resp.\ $H_{k}(\underline{a};M)$) denote the $k^{\textnormal{th}}$ homology group of $K(\underline{a})$ (resp.\ $K(\underline{a};M)$). For an integer $i>0$, put ${\underline{a}^{i}}=a^{i}_{1},\ldots, a^{i}_{n}$. 

For $a\in A$ and positive integers $i<j$, 
we define a morphism of complexes $K(a^{j})\to K(a^{i})$ which is identity in degree $0$, and multiplication by $a^{j-i}$ in degree $1$. It induces a morphism of complexes $K(\underline{a}^{j})\to K(\underline{a}^{i})$ (resp.\ $K(\underline{a}^{j};M)\to K(\underline{a}^{i};M)$) and a homomorphism $H_{k}(\underline{a}^{j})\to H_{k}(\underline{a}^{i})$ (resp.\ $H_{k}(\underline{a}^{j};M)\to H_{k}(\underline{a}^{i};M)$) in the natural way, such that $\{K(\underline{a}^{i})\}_{i}$ and $\{H_{k}(\underline{a}^{i})\}_{i}$ (resp.\ $\{K(\underline{a}^{i};M)\}_{i}$ and $\{H_{k}(\underline{a}^{i};M)\}_{i}$) form inverse systems. 
\begin{definition}
\label{wkprrdf}
\normalfont
We say that a sequence $\underline{a}=a_{1},\ldots, a_{n}$ in $A$ is \emph{weakly proregular}, if the 
inverse system $\{H_{k}(\underline{a}^{i})\}_{i}$ is pro-zero (i.e.\ for every integer $n>0$, there exists an integer $m>n$ such that the homomorphism $H_{k}(\underline{a}^{m})\to H_{k}(\underline{a}^{n})$ is the zero map) for every $k>0$. An ideal $I\subset A$ is said to be \emph{weakly proregular}, if it is generated by a weakly proregular sequence. 
\end{definition}
Note that for finitely generated ideals $I, J\subset A$ defining the same adic topology, $I$ is weakly proregular if and only if $J$ is so (\cite[Corollary 6.4]{PSY}). 
\begin{remark}
\label{wkprrmk}
\normalfont
For a principal ideal $I=tA$, $I$ is weakly proregular if and only if $A$ has bounded $I$-torsion. Indeed, if $I$ is weakly proregular, there exists an integer $l>0$ such that $t^{l-1}\textnormal{Ann}_{A}(t^{l})=(0)$. Then $\textnormal{Ann}_{A}(t^{l+n+1})=\textnormal{Ann}_{A}(t^{l+n})$ for every integer $n\geq 0$ (otherwise, $\textnormal{Ann}_{A}(t^{l+n+1})\setminus\textnormal{Ann}_{A}(t^{l+n})$ would have an element $x$ for some $n$, such that $t^{n+1}x\in \textnormal{Ann}_{A}(t^{l})$ and $t^{l-1}(t^{n+1}x)\neq 0$). Hence $A$ has bounded $I$-torsion. The converse is clear. Moreover, in this case, $A$ has bounded $I$-torsion if and only if $(A, I)$ satisfies $(*)$ (it can be checked by \cite[Theorem 7.4(iii)]{Mat} like Remark \ref{ebea}). 
\end{remark}
Any ideal of a noetherian ring is weakly proregular (\cite[Theorem 4.34]{PSY}). More generally, for a pair $(A, I)$ which satisfies $(*)$, $I$ is weakly proregular. 
Indeed, $\bold{(E_{a})}$ implies the weak proregularity as remarked above. If $(A, I)$ satisfies $\bold{(E_{b})}$, there exist a noetherian ring $A_{0}$ and elements $t_{1},\ldots, t_{r}\in A_{0}$ such that $A$ is flat over $A_{0}$ and $I=(t_{1},\ldots, t_{r})A$. 
Set $t_{0}=0\ (\in A_{0})$. For every integer $n>0$ and each $i=0,\ldots,r-1$, an ascending chain of ideals 
\begin{eqnarray}
[(t^{n}_{0},\ldots, t^{n}_{i})A:t_{i+1}]\subset[(t^{n}_{0},\ldots, t^{n}_{i})A:t^{2}_{i+1}]\subset[(t^{n}_{0},\ldots, t^{n}_{i})A:t^{3}_{i+1}]\subset\cdots\nonumber
\end{eqnarray}
stabilizes by noetherianness of $A_{0}$ and flatness of $A_{0}\to A$ (\cite[Theorem 7.4(i$\mspace{-1mu}$i$\mspace{-1mu}$i)]{Mat}). Thus there exists $m>n$ such that $[(t^{m}_{0},\ldots, t^{m}_{i})A:t^{m}_{i+1}]\subset [(t^{n}_{0},\ldots, t^{n}_{i})A:t^{m-n}_{i+1}]$ 
(i.e.\ the sequence $t_{1},\ldots, t_{r}$ is \emph{proregular}, see \cite[Definition 2.6]{Szl}). Hence $I$ is weakly proregular by \cite[Lemma 2.7]{Szl}. 

On the other hand, in all our counter-examples, the defining ideals cannot be weakly proregular because of the failure of the approximation. It is clear from Remark \ref{wkprrmk} that the ideal $tA$ in Example \ref{principalex} is not weakly proregular. For Example \ref{boundedtorex}, the statement can be checked as follows. 
\begin{claim}
\label{nwkpr}
In Example \ref{boundedtorex}, the ideal $I\subset A'$ is not weakly proregular. 
\end{claim}
\begin{proof}
Assume that $I=(t, u)A'$ is weakly proregular. Then, the sequence $t, u$ (in $A'$) is weakly proregular (\cite[Corollary 6.3]{PSY}). Note that for $i<j$ the morphisms $K(t^{j})\to K(t^{i})$ and $K(u^{j})\to K(u^{i})$ induce the following commutative diagram with exact rows (cf.\ \cite[p.167]{Szl} and \cite[Lemma 4.5.3]{Wei}):
\begin{eqnarray}
\begin{CD}\nonumber
0 @>>> H_{0}(u^{j};H_{1}(t^{j})) @>>> H_{1}(t^{j}, u^{j}) @>>> H_{1}(u^{j};H_{0}(t^{j})) @>>> 0\\
@. @VVV @VVV @VVV @. \\
0 @>>> H_{0}(u^{i};H_{1}(t^{i})) @>>> H_{1}(t^{i}, u^{i}) @>>> H_{1}(u^{i};H_{0}(t^{i})) @>>> 0 @. \ .
\end{CD}
\end{eqnarray}
Here by assumption $\{H_{1}(t^{i}, u^{i})\}_{i}$ is pro-zero, and thus
so is the inverse system $\{H_{0}(u^{i};H_{1}(t^{i}))\}_{i}$. Hence there exists an integer $l>2$ such that 
for every integer $\nu\geq l$ the 
transition map from $H_{0}(u^{\nu};H_{1}(t^{\nu}))$ to $H_{0}(u^{2};H_{1}(t^{2}))$ is the zero map. Since $x_{\nu-2}\in H_{1}(t^{\nu})$, there exists $z_{\nu}\in H_{1}(t^{2})$ such that $t^{\nu-2}x_{\nu-2}=u^{2}z_{\nu}$. 
Here $t^{\nu-2}x_{\nu-2}$ is $u$-torsion, and therefore so is $z_{\nu}$. Hence $z_{\nu}$ is $I$-torsion. Thus $t^{\nu-2}x_{\nu-2}=u^{2}z_{\nu}=0$ by the proof of Claim \ref{btcl5}. Then the solution $\widehat{\alpha}=\sum^{\infty}_{i=0}x_{i}t^{i}\in\widehat{A}'$ must come from $A'$. This is absurd. 
\end{proof}
\section*{Acknowledgements}
The author would like to thank Professor 
Kazuhiro Fujiwara for careful advice. He is also grateful to Professor Ryo Takahashi for useful comments. His special appreciation goes to Dr.\ Liran Shaul for pointing out that the primitive version of this paper is concerned with weak proregularity. Finally, he would like to thank the referee for careful reading and the valuable comments. 


(Kei Nakazato)\\
\textsc{Graduate School of Mathematics\\Nagoya University\\Nagoya 464-8602, Japan}\\
\emph{E-mail}:
\texttt{m11047c@math.nagoya-u.ac.jp}

\begin{thebibliography}{99}


\bibitem{AJL} L.\ Alonso Tarr\'{i}o, A.\ Jerem\'{i}as L\'{o}pez, J.\ Lipman, Local homology and cohomology on schemes, \emph{Ann.\ Sci.\ \'{E}c.\ Norm.\ Sup\'{e}r. (4)} \textbf{30} (1997), 1-39. Correction, available online at \\
\texttt{http://www.math.purdue.edu/\~{}lipman/papers/homologyfix.pdf}
\bibitem{And} M.\ Andr\'{e}, \emph{Cinq expos\'{e}s sur la d\'{e}singularisation}, Handwritten manuscript, \'{E}cole Polytechnique F\'{e}d\'{e}rale de Lausanne, 1991. 
\bibitem{YAn} Y.\ Andr\'{e}, La conjecture du facteur direct, \emph{Publ.\ Math.\ Inst.\ Hautes\ \'{E}tudes\ Sci.} (2018), \texttt{https://doi.org/10.1007/s10240-017-0097-9}, in press. 
\bibitem{Art1} M.\ Artin, {Algebraic approximation of structures over complete local rings}, \emph{Publ.\ Math.\ Inst.\ Hautes\ \'{E}tudes\ Sci.}\ \textbf{36} (1969), 23-58. 
\bibitem{Elk} R.\ Elkik, {Solutions d'\'{e}quations \`{a} coefficients dans un anneau hens\'{e}lien}, \emph{Ann.\ Sci.\ \'{E}c.\ Norm.\ Sup\'{e}r.\ (4)} \textbf{6} (1973), 553-603. 
\bibitem{FGK} K.\ Fujiwara, O.\ Gabber, F.\ Kato, {On Hausdorff completions of commutative rings in rigid geometry}, \emph{J.\ Algebra} \textbf{332} (2011), 293-321. 
\bibitem{Gre} S.\ Greco, {Henselization of a ring with respect to an ideal}, \emph{Trans.\ Amer.\ Math.\ Soc.}\ \textbf{144} (1969), 43-65. 
\bibitem{GS} S.\ Greco, P.\ Salmon, \emph{Topics in ${\mathfrak{m}}$-adic Topologies}, 
Ergeb.\ Math.\ Grenzgeb.\ (2), vol.\ {58}, Springer-Verlag, Berlin, 1971. 
\bibitem{Gro} A.\ Grothendieck, \emph{Local Cohomology}, Notes by R.\ Hartshorne, 
Lecture Notes in Math.,\ vol.\ {41}, 1967. 
\bibitem{KPR} H.\ Kurke, G.\ Pfister, M.\ Roczen, \emph{Henselsche Ringe und Algebraische Geomet-rie}, VEB Deutscher Verlag der Wissenschaften, Berlin, 1975. 
\bibitem{Mat} H.\ Matsumura, \emph{Commutative Ring Theory}, Cambridge Stud.\ Adv.\ Math.,\ vol.\ {8}, Cambridge University Press, Cambridge, 1989. 
\bibitem{Mor} L.\ Moret-Bailly, {An extension of Greenberg's theorem to general valuation rings}, \emph{Manuscripta\ Math.}\ \textbf{139} (2012), 153-166. 
\bibitem{Ogo}T.\ Ogoma, {General N\'{e}ron desingularization based on the idea of Popescu}, \emph{J.\ Algebra} \textbf{167} (1994), 57-84. 
\bibitem{P85} D.\ Popescu, {General N\'{e}ron desingularization}, \emph{Nagoya Math.\ J.}\ \textbf{100} (1985), 97-126. 
\bibitem{P86} D.\ Popescu, {General N\'{e}ron desingularization and approximation}, \emph{Nagoya Math. J.}\ \textbf{104} (1986), 85-115. 
\bibitem{P90}D.\ Popescu, {Letter to the editor:\ ``General N\'{e}ron desingularization and approximation''}, \emph{Nagoya Math.\ J.}\ \textbf{118} (1990), 45-53. 
\bibitem{PSY} M.\ Porta, L.\ Shaul, A.\ Yekutieli, On the homology of completion and torsion, \emph{Algebr.\ Represent.\ Theory} \textbf{17} (2014), 31-67. 
\bibitem{Ray} M.\ Raynaud, \emph{Anneaux Locaux Hens\'{e}liens}, Lecture Notes in Math.,\ vol.\ {169}, Springer-Verlag, Berlin, 1970. 
\bibitem{Szl} P.\ Schenzel, Proregular sequences, local cohomology, and completion, \emph{Math.\ Scand.\ }\textbf{92} (2003), 161-180. 
\bibitem{Sch} H.\ Schoutens, {Approximation properties for some non-Noetherian local rings}, \emph{Pacific J.\ Math.}\ \textbf{131} (1988), 331-359. 
\bibitem{Shi} K.\ Shimomoto, An application of the almost purity theorem to the homological conjectures, \emph{J.\ Pure Appl.\ Algebra} \textbf{220} (2016), 621-632. 
\bibitem{Swa} R.\ Swan, {N\'{e}ron-Popescu desingularization}, in:\ \emph{Proceedings of the International Conference of Algebra and Geometry (Taipei, 1995)}, International Press, Cambridge, 1998, pp.\ 135-192. 
\bibitem{Wei} C.A.\ Weibel, \emph{An Introduction to Homological Algebra}, Cambridge Stud.\ Adv.\ Math.,\ vol.\ {38}, Cambridge University Press, Cambridge, 1994. 
\end{thebibliography}
\end{document}